\documentclass[a4paper, 10pt]{article}
\usepackage{mathrsfs,amsmath,amsthm}
\usepackage{graphicx, color, subfigure}
\usepackage{array}
\usepackage{cases}
\makeatletter
\def\th@plain{%
  \upshape 
}
\makeatother

\makeatletter
\renewenvironment{proof}[1][\proofname]{\par
  \pushQED{\qed}%
  \normalfont \topsep6\p@\@plus6\p@\relax
  \trivlist
  \item[\hskip\labelsep
        \bfseries
    #1\@addpunct{.}]\ignorespaces
}{%
  \popQED\endtrivlist\@endpefalse
}
\makeatother

\newtheorem{theorem}{Theorem}
\numberwithin{theorem}{section}
\newtheorem{lemma}{Lemma}
\newtheorem{corollary}{Corollary}

\newtheorem*{conjecture*}{Conjecture}

\newtheorem{claim}{Claim}

\newtheorem{observation}{Observation}

\theoremstyle{definition}

\usepackage[
  centering,  
  headheight=50pt,
  headsep=12pt,
  footskip=12pt,
  footnotesep=24pt plus 2pt minus 12pt]{geometry}
\usepackage[T1]{fontenc}
\usepackage[varg]{txfonts}

\usepackage{enumerate}
\usepackage[square, numbers, sort&compress]{natbib}

\ifx\pdfoutput\undefined
 \usepackage[dvipdfm,%
  bookmarks=true,%
  bookmarksnumbered=true, 
  bookmarksopen=true, 
  plainpages=false,%
  pdfpagelabels,%
  colorlinks=true, 
  linkcolor=black, 
  citecolor=black,%
  hyperindex=true,
  urlcolor=black,
  pdfborder=001]{hyperref}
\else
 \usepackage[pdftex,%
  bookmarks=true,%
  bookmarksnumbered=true, 
  bookmarksopen=true, 
  plainpages=false,%
  pdfpagelabels,%
  colorlinks=true, 
  linkcolor=black, 
  citecolor=black,%
  anchorcolor=green,
  urlcolor= blue,
  breaklinks=true，
  hyperindex=true,
  pdfborder=001]{hyperref}
\fi
\hypersetup{%
  pdfstartview=FitH, 
  pdfauthor={Tao Wang}}

\newcommand{\etal}{et~al.\ }
\newcommand{\ie}{i.e.,\ }

\def\int(#1){\mathrm{int}(#1)}
\def\ext(#1){\mathrm{ext}(#1)}
\def\Int(#1){\mathrm{Int}(#1)}
\def\Ext(#1){\mathrm{Ext}(#1)}

\DeclareMathOperator {\diam}{diam}

\graphicspath{{figures/}}
\numberwithin{figure}{section}
\numberwithin{equation}{section}

\newtheorem*{asseration}{Asseration}
\newtheorem*{Murty-Simon}{Murty-Simon Conjecture}
\newtheorem{remark}{Remark}
\begin{document}%
\title{On Murty-Simon Conjecture II}
\author{Tao Wang\unskip\textsuperscript{a,}\footnote{{\tt Corresponding
author: wangtao@henu.edu.cn}}, \ \  Ping Wang\unskip\textsuperscript{b},
\ \ Qinglin Yu\unskip\textsuperscript{c}\\[.5em]
{\small \textsuperscript{a}\unskip Institute of Applied Mathematics}\\
{\small College of Mathematics and Information Science}\\
{\small Henan University, Kaifeng, 475004, P. R. China}\\
{\small \textsuperscript{b}\unskip Department of Mathematics, Statistics and Computer Science}\\
{\small St. Francis Xavier University, Antigonish, NS, Canada}\\
{\small \textsuperscript{c}\unskip Department of Mathematics and Statistics}\\
{\small Thompson Rivers University, Kamloops, BC, Canada}
}
\date{}
\maketitle

\begin{abstract}%
A graph is diameter two edge-critical if its diameter is two and the deletion of any edge increases the diameter. Murty and Simon conjectured that the number of edges in a diameter two edge-critical graph on $n$ vertices is at most $\left \lfloor \frac{n^{2}}{4} \right \rfloor$ and the extremal graph is the complete bipartite graph $K_{\left\lfloor
\frac{n}{2} \right\rfloor, \left\lceil \frac{n}{2} \right\rceil}$. In the series papers \cite{MR2818869, MR2811636, MR2852589}, the Murty-Simon Conjecture stated by Haynes \etal is
not the original conjecture, indeed, it is only for the diameter two edge-critical graphs of even
order. In this paper, we completely prove the Murty-Simon Conjecture for the graphs whose
complements have vertex connectivity $\ell$, where $\ell = 1, 2, 3$; and for the graphs whose complements have an independent vertex cut of cardinality at least three.

Keywords: Murty-Simon Conjecture; diameter two edge critical graph; total domination edge critical graph
\end{abstract}

\section{Introduction}
All graphs considered in this paper are simple. Let $G$ be a graph with vertex set $V(G)$ and edge set $E(G)$. The
{\em neighborhood} of a vertex $v$ in a graph $G$, denoted by
$N_{G}(v)$, is the set of all the vertices adjacent to the vertex
$v$, \ie $N_{G}(v) = \{u \in V(G) \mid uv \in E(G)\}$, and the {\em
closed neighborhood} of a vertex $v$ in $G$, denoted by $N_{G}[v]$,
is defined by $N_{G}[v] = N_{G}(v) \cup \{v\}$. For a subset $S
\subseteq V$, the {\em neighborhood of the set $S$} in $G$ is the
set of all vertices adjacent to vertices in $S$, this set is denoted
by $N_{G}(S)$, and the {\em closed neighborhood of $S$} by $N_{G}[S]
= N_{G}(S) \cup S$. Let $S$ and $T$ be two subsets (not necessarily
disjoint) of $V(G)$, $[S, T]$ denotes the set of edges of $G$
with one end in $S$ and the other in $T$, and $e_{G}(S, T)=|[S,
T]|$. If every vertex in $S$ is adjacent to each vertex in $T$, then we say that $[S, T]$ is full. If $S \subseteq V(G)$, and $u, v$ are two nonadjacent vertices in $G$, then we say that $xy$ is a {\em missing edge} in $S$ (rather than ``$uv$ is a missing edge in $G[S]$'').

The {\em complement} $G^{c}$ of a simple graph $G = (V, E)$ is the
simple graph with vertex set $V$, two vertices are adjacent in
$G^{c}$ if and only if they are not adjacent in $G$.

Given a graph $G$ and two vertices $u$ and $v$ in it, the {\em distance}
between $u$ and $v$ in $G$, denoted by $d_{G}(u, v)$, is the length
of a shortest $u$-$v$ path in $G$; if there is no path connecting
$u$ and $v$, we define $d_{G}(u, v) = \infty$. The {\em diameter} of
a graph $G$, denoted by $\diam(G)$, is the maximum distance between any two vertices of $G$. Clearly, $\diam(G) = \infty$ if and only if $G$ is disconnected.

A subset $S \subseteq V$ is called a {\em dominating set} ({\bf DS})
of a graph $G$ if every vertex $v \in V$ is an element of $S$ or is
adjacent to a vertex in $S$, that is, $N_{G}[S] = V$. The {\em
domination number} of $G$, denoted by $\gamma(G)$, is the minimum
cardinality of a dominating set in $G$.

A subset $S \subseteq V$ is a {\em total dominating set},
abbreviated {\bf TDS}, of $G$ if every vertex in $V$ is adjacent to
a vertex in $S$, that is $N_{G}(S) = V$. Every graph without
isolated vertices has a TDS, since $V$ is a trivial TDS. The
{\em total domination number} of a graph $G$, denoted by
$\gamma_{t}(G)$, is the minimum cardinality of a TDS in $G$. For the graph with isolated vertices, we define its total domination number to be $\infty$. Total
domination in graphs was introduced by Cockayne, Dawes, and Hedetniemi
\cite{MR0584887}.

\begin{observation}\label{TDSCAP}%
Let $G$ be a graph, for any vertex $v$ and a TDS $S$ in $G$. Then $S
\cap N_{G}[v] \neq \emptyset$.
\end{observation}

For two vertex subsets $X$ and $Y$, we say that {\em $X$ dominates $Y$} ({\em totally dominates $Y$}, respectively) if $Y \subseteq N_{G}[X]$ ($Y \subseteq N_{G}(X)$, respectively); sometimes, we also say that {\em $Y$ is dominated by $X$} ({\em totally dominated by $X$}, respectively).

For three vertices $u, v, w\in V(G)$, the symbol $uv\rightarrow w$ means that $\{u, v\}$ dominates $G - w$,
but $uw \notin E(G)$, $vw \notin E(G)$ and $uv \in E(G)$.

A graph $G$ is said to be {\em diameter-$d$ edge-critical} if $\diam(G) = d$
and $\diam(G - e) > \diam(G)$ for any edge $e \in E(G)$. Gliviak \cite{MR0412035} proved the impossibility of characterization of diameter-$d$ edge-critical graphs by finite extension or by forbidden subgraphs. Plesn\'{i}k
\cite{Plesn'ik1975} observed that all known minimal graphs of
diameter two on $n$ vertices have no more than $\left \lfloor
\frac{n^{2}}{4} \right \rfloor $ edges. Independently, Murty and
Simon (see \cite{MR0548621}) conjectured the following:
\begin{Murty-Simon}\label{MurtySimonConjecture}
If $G$ is a diameter-$2$ edge-critical graph on $n$ vertices, then $|E(G)| \leq \left
\lfloor \frac{n^{2}}{4} \right \rfloor$. Moreover, equality holds if
and only if $G$ is the complete bipartite graph $K_{\left\lfloor
\frac{n}{2} \right\rfloor, \left\lceil \frac{n}{2} \right\rceil}$.
\end{Murty-Simon}

Let $G$ be a diameter-$2$ edge-critical graph on $n$ vertices. Plesn\'{i}k
\cite{Plesn'ik1975} proved that $|E(G)| < 3n(n-1)/8$. Caccetta and H\"{a}ggkvist \cite{MR0548621} obtained that $|E(G)| < 0.27n^{2}$. Fan \cite{MR0915948} proved  the first part of the Murty-Simon Conjecture for $n \leq 24$ and for $n=26$; and 
\[
|E(G)| < \frac{1}{4}n^{2} + (n^{2}-16.2n + 56)/320 < 0.2532 n^{2}\]
for $n \geq 25$.
F\"{u}redi \cite{MR1147806} proved the Murty-Simon Conjecture for $n > n_{0}$, where $n_{0}$ is not larger than a tower of $2$'s of height about $10^{14}$.

A graph is {\em total domination edge critical} if the addition of any edge decrease the total domination number. If $G$ is total domination edge critical with $\gamma_{t}(G) = k$, then we say that $G$ is a {\em $k$-$\gamma_{t}$-edge critical graph}. Haynes \etal \cite{MR1658130} proved that the addition of an edge to a graph without isolated vertices can decrease the total domination number by at most two. A graph $G$ with the property that $\gamma_{t}(G) = k$ and $\gamma_{t}(G + e) = k-2$ for every missing edge $e$ in $G$ is called a {\em $k$-supercritical} graph.

\begin{theorem}[Hanson and Wang \cite{MR1976247}]\label{DominationEdgeVsDiameterTwo}
A nontrivial graph $G$ is dominated by two adjacent vertices if and only if the
diameter of $G^{c}$ is greater than two.
\end{theorem}

\begin{corollary}%
A graph $G$ is diameter-$2$ edge-critical on $n$ vertices if and only if the total domination number of $G^{c}$ is greater than two but the addition of any edge in $G^{c}$ decrease the total domination number to be two, that is, $G^{c}$ is $K_{1} \cup K_{n-1}$ or $3$-$\gamma_{t}$-edge critical or $4$-supercritical.
\end{corollary}

The complement of $G$ is $K_{1} \cup K_{n-1}$ if and only if $G$ is $K_{1, n-1}$. Clearly, the Murty-Simon Conjecture holds for $K_{1, n-1}$.

The $4$-supercritical graphs are characterized in \cite{MR1676474}.

\begin{theorem}
A graph $H$ is $4$-supercritical if and only if $H$ is the disjoint union of two nontrivial complete graphs. 
\end{theorem}

The complement of a $4$-supercritical graph is a complete bipartite graph. The Murty-Simon Conjecture holds for the graphs whose complements are $4$-supercritical, \ie complete bipartite graphs.

Therefore, we only have to consider the graphs whose complements are $3$-$\gamma_{t}$-edge critical. 

For $3$-$\gamma_{t}$-edge critical graphs, the bound on the diameter is established in \cite{MR1658130}.
\begin{theorem}
If $G$ is a $3$-$\gamma_{t}$-edge critical graph, then $2 \leq \diam(G) \leq 3$.
\end{theorem}

Hanson and Wang \cite{MR1976247}
partition the family of $3$-$\gamma_{t}$-edge critical graphs into
two classes in terms of the diameter:
\[
\mathfrak{G}_{3} = \left \{ G \mid G \mbox{ is a } 3\mbox{-}\gamma_{t}\mbox{-edge critical graph on }n \mbox{ vertices and }\diam(G) = 3 \right\}
\]
\[
\mathfrak{G}_{2} = \left \{ G \mid G \mbox{ is a } 3\mbox{-}\gamma_{t}\mbox{-edge critical graph on }n \mbox{ vertices and }\diam(G) = 2 \right\}
\]
and proved the first part of the Murty-Simon Conjecture for the graphs whose complement are in $\mathfrak{G}_{3}$. Recently, Haynes, Henning, van der Merwe and Yeo \cite{MR2818869} proved the second part for the graphs whose complements are $3$-$\gamma_{t}$-edge critical graphs with diameter three but only with even vertices. Also, Haynes \etal \cite{MR2852589} proved the Murty-Simon Conjecture for the graphs of even order whose complements have vertex connectivity $\ell$, where $\ell = 1, 2, 3$. Haynes, Henning and Yeo \cite{MR2811636} proved the Murty-Simon Conjecture for the graphs whose complements are claw-free.

\begin{remark}%
In the series papers \cite{MR2818869, MR2811636, MR2852589}, the Murty-Simon Conjecture stated by Haynes \etal is not the original conjecture, indeed, it is only for the diameter two edge critical graphs of even order. In this paper, we completely prove the Murty-Simon Conjecture for the graphs whose complements have vertex connectivity $\ell$, where $\ell = 1, 2, 3$.
\end{remark}

Let $G$ be a $3$-$\gamma_{t}$-edge critical graph. Then the addition of any edge $e$ decrease the total domination number to be two, that is, $G + e$ is dominated by two adjacent vertices $x$ and $y$; we call such edge $xy$ {\em quasi-edge} of $e$. Note that $xy$ must contain at least one end of $e$. Clearly, quasi-edge of $e$ may not be unique. If $xy \mapsto w$, then $xy$ is quasi-edge of the missing edge $xw$, and also quasi-edge of missing edge $yw$; conversely, if $xy$ is quasi-edge of a missing edge, then there exists an unique vertex $w$ such that $xy \mapsto w$.

From the definition of $3$-$\gamma_{t}$-edge critical graph, we have the following frequently used observation.
\begin{observation}%
If $G$ is a $3$-$\gamma_{t}$-edge critical graph and $uv$ is a missing edge in it, then either
\begin{enumerate}[(i)]
\item $\{u, v\}$ dominates $G$; or
\item there exists a vertex $z$ such that $uz \mapsto v$ or $zv \mapsto u$.
\end{enumerate}
\end{observation}

\section{Main results}
The following fundamental result was observed by Hanson and Wang \cite{MR1976247}, also formally written by Haynes \etal \cite{MR2818869}.
\begin{lemma}\label{QuasiClique}
Let $G$ be a graph on $n$ vertices and $(A, B)$ be a partition of its vertex set. If we can associate every missing edge in $A$ or $B$ with an edge in $[A, B]$ and this association is unique in sense that no two missing edges in $A$ or $B$ can associate with one edge in $[A, B]$, then $|E(G^{c})| \leq |A| \times |B| \leq \left\lfloor \frac{n^{2}}{4} \right\rfloor$. Moreover, if there exists an additional edge in $[A, B]$ which is not associate with any missing edge in $A$ or $B$, then $|E(G^{c})| <\left\lfloor \frac{n^{2}}{4} \right\rfloor$. 
\end{lemma}

The following lemma is extracted from the proof in \cite{MR2818869}, but for the sake
of completeness we present here a full self-contained proof.
\begin{lemma}\label{PartitionEdge} 
Let $G$ be a $3$-$\gamma_{t}$-edge critical graph on $n$ vertices, and $(A, B)$ be a
partition of the vertex set $V(G)$. If, for every missing edge $e$ in $A$ and $B$, there exists a quasi-edge of $e$ in $[A, B]$, then $|E(G^{c})| \leq |A| \times |B| \leq \left\lfloor \frac{n^{2}}{4} \right\rfloor$. Moreover, if $|E(G^{c})| = \left\lfloor \frac{n^{2}}{4} \right\rfloor$, then we have the following properties:
\begin{enumerate}[(i)]
\item For every missing edge $e$ in $A$ and $B$, there exists precisely one quasi-edge of $e$ in $[A, B]$; conversely, for every edge in $[A, B]$, it is the quasi-edge of a missing edge in $A$ or $B$.
\item If $u_{1}, u_{2} \in A$ and $v_{1}, v_{2} \in B$, $\{u_{1}v_{1}, u_{2}v_{2}\} \subseteq E(G^{c})$ and $\{u_{1}v_{2}, u_{2}v_{1}\} \subseteq E(G)$, then $\{u_{1}u_{2}, v_{1}v_{2}\} \subseteq E(G)$.
\item If $u_{1}u_{2}$ is a missing edge in $A$ and $\deg_{B}(u_{1}) \geq \deg_{B}(u_{2})$, then $N_{B}(u_{1}) = N_{B}(u_{2}) \cup \{y\}$, where $y$ is the end (in $B$) of the quasi-edge of $u_{1}u_{2}$. Similarly, if $v_{1}v_{2}$ is a missing edge in $B$ and $\deg_{A}(v_{1}) \geq \deg_{A}(v_{2})$, then $N_{A}(v_{1}) = N_{A}(v_{2}) \cup \{x\}$, where $x$ is the end (in $A$) of the quasi-edge of $v_{1}v_{2}$. Consequently, the missing edges in $A$ (resp. in $B$) form a bipartite graph on $A$ (resp. on $B$).
\end{enumerate}
\end{lemma}
\begin{proof}%
Suppose that $uv$ is a missing edge in $A$, by the
hypothesis, without loss of generality, there exists an edge $uw$ in $[A, B]$ such that $uw \mapsto v$. Clearly, $v$ is not
dominated by $\{u, w\}$, and thus for any missing edge $e \neq uv$ in $A$ or $B$, the edge $uw$ is not a quasi-edge of $e$. Hence, for distinct missing edges $e$ and
$e'$ in $A$, they have no common quasi-edges in $[A, B]$.
Similarly, for distinct missing edges $e$ and $e'$ in $B$, they have no common quasi-edges in $[A, B]$.

It is easy to check that for any missing edge $e$ in $A$ and missing edge $e'$ in $B$, they have no common quasi-edges in $[A, B]$. We associate every missing edge in $A$ and $B$ with its quasi-edge in $[A, B]$, by Lemma \ref{QuasiClique}, it follows that $|E(G^{c})| \leq |A| \times |B| \leq \left \lfloor
\frac{n^{2}}{4} \right \rfloor$.

(ii) If $u_{1}u_{2} \notin E(G)$, then both $u_{1}v_{2}$ and $u_{2}v_{1}$ are quasi-edge of $u_{1}u_{2}$, a contradiction. Similarly, we can prove that $v_{1}v_{2} \in E(G)$.

(iii) Let $u_{1}u_{2}$ be a missing edge in $A$. Suppose that $N_{B}(u_{1}) \nsubseteqq N_{B}(u_{2})$ and $N_{B}(u_{2}) \nsubseteqq N_{B}(u_{1})$. Choose a vertex $v_{1} \in N_{B}(u_{2}) \setminus N_{B}(u_{1})$ and a vertex $v_{2}$ in $N_{B}(u_{1}) \setminus N_{B}(u_{2})$, then $\{u_{1}v_{1}, u_{2}v_{2}\} \subseteq E(G^{c})$ and $\{u_{1}v_{2}, u_{2}v_{1}\} \subseteq E(G)$, by (ii), we have $u_{1}u_{2} \in E(G)$, a contradiction. Hence $N_{B}(u_{1}) \supseteq N_{B}(u_{2})$. If $|N_{B}(u_{1}) \setminus N_{B}(u_{2})| \geq 2$, then there are at least two quasi-edge of the missing edge $u_{1}u_{2}$, a contradiction. Therefore, $N_{B}(u_{1}) = N_{B}(u_{2}) \cup \{y\}$. Similarly, we can prove that $N_{A}(v_{1}) = N_{A}(v_{2}) \cup \{x\}$, if $v_{1}v_{2}$ is a missing edge in $B$.

In the graph formed by the missing edges in $A$, one part $X$ is the vertices of degree odd in $B$, and the other part $Y$ is the vertices of degree even in $B$. For any missing edge $uv$, $\deg_{B}(u)$ and $\deg_{B}(v)$ differ by exactly one, so one is odd and the other is even, and hence $uv$ has one end in $X$ and the other in $Y$, then the graph is bipartite. Similarly, the graph formed by the missing edges in $B$ is a bipartite graph.
\end{proof}
To settle the Murty-Simon Conjecture, the remaining graphs to be verified are ones whose complements are in $\mathfrak{G}_{2}$. We show that the conjecture holds if a condition
in terms of independent cut is satisfied.
\begin{theorem}
Let $G$ be a $3$-$\gamma_{t}$-edge critical graph on $n$ vertices with $\delta(G) \geq 3$. If $G$ has an
independent vertex cut of cardinality at least three, then $|E(G^{c})|
< \left \lfloor \frac{n^{2}}{4} \right \rfloor$.
\end{theorem}
\begin{proof}%
First, we prove the following asseration:
\begin{asseration}
There exists an independent vertex cut $S$ and a component $K$ of $G - S$ such that every vertex in $K$ dominates $V(K) \cup S$.
\end{asseration}
\begin{proof}%
If $G$ has a vertex $v$ such that $N_{G}(v)$ is independent, then let $S = N_{G}(v)$ and $K = \{v\}$, we are done. So we may assume that there is no such vertex. Let $S$ be an independent vertex cut of cardinality at least three and $G_{1}$ be a
component of $G - S$, and $G_{2}$ be the union of the other
components. Moreover, by the above argument, we may assume that $|V(G_{1})| \geq 2$ and $|V(G_{2})| \geq 2$.

Assume that there exists a vertex $v \in S$, a vertex $w_{1} \in V(G_{1})$ and a vertex
$w_{2} \in V(G_{2})$ such that $\{vw_{1}, vw_{2}\} \subseteq E(G^{c})$.
Since $\{w_{1}, w_{2}\}$ does not dominate $v$, by \autoref{TDSCAP}, there exists a vertex $w$ such that either $w_{1}w$ is an edge and $w_{1}w \mapsto w_{2}$
or $ww_{2}$ is an edge and $ww_{2} \mapsto w_{1}$. In the former case, $w \in V(G_{1})$, but $\{w_{1}, w\}$ does not dominate $V(G_{2}) \setminus \{w_{2}\}$ for $|V(G_{2})| \geq 2$, a contradiction; a similar contradiction
can be obtained for the latter case. Therefore, for every vertex $v$ in $S$, it dominates $G_{1}$ or $G_{2}$.

Suppose that $G_{1}$ is not dominated by $v_{1} \in S$ and
$G_{2}$ is not dominated by $v_{2} \in S$. By the previous argument, $v_{1} \neq v_{2}$ and thus $v_{1}$ dominates $G_{2}$ and $v_{2}$ dominates $G_{1}$. Since $\{v_{1}, v_{2}\}$ does not dominate $S\setminus \{v_{1}, v_{2}\}$ (note that $S\setminus \{v_{1}, v_{2}\} \neq \emptyset$ for $S$ is independent and $|S| \geq 3$), by \autoref{TDSCAP}, there exists a vertex $v'$ such that $v_{1}v'$ is an edge and
$v_{1}v' \mapsto v_{2}$ or $v'v_{2}$ is an edge and $v'v_{2} \mapsto v_{1}$. Without loss of generality, assume that
the former case holds. Since $v_{1}$ does not dominates $G_{1}$,
the vertex $v'$ must be in $V(G_{1})$ and $v'v_{2} \notin E(G)$, which contradicts the fact that $v_{2}$ dominates $G_{1}$. Therefore, without loss of generality, we may assume that $G_{1}$ is dominated by every vertex in $S$.

Next, we show that $G_{1}$ is complete.
Otherwise, there is a missing edge $u_{1}u_{2}$ in
$V(G_{1})$. Since $\{u_{1}, u_{2}\}$ does not dominate $G_{2}$, 
there exists a vertex $u$ such that $u_{1}u \mapsto u_{2}$ or $uu_{2} \mapsto u_{1}$. Without loss of generality, assume that $u_{1}u \mapsto u_{2}$. We have known $u_{1}u \in E(G)$ and $uu_{2} \notin E(G)$, so $u \in V(G_{1})$. But $\{u_{1}, u\}$ does not dominate $V(G_{2})$, which is a contradiction. Therefore, $G_{1}$ is complete. Let $K = G_{1}$, we complete the proof.
\end{proof}

Let $A = V(K) \cup S$ and $B = V\setminus A$. For any missing edge $xy$ in $A$, indeed $xy$ is a missing edge in $S$ by the asseration. Both $x$
and $y$ dominates $K$, the quasi-edge of $xy$ must
have one end in $B$, \ie its quasi-edges lies in $[A, B]$. For any missing edge $x'y'$ in
$B$, the closed neighborhood of every vertex
in $K$ is contained in $A$, then quasi-edges of $x'y'$ must have one end in $A$. By \autoref{PartitionEdge}, 
\begin{equation}%
|E(G^{c})| \leq |A| \times |B| \leq \left \lfloor \frac{n^{2}}{4} \right \rfloor.
\end{equation}

If $|E(G^{c})| = \left \lfloor \frac{n^{2}}{4} \right \rfloor$, then the missing edges in $A$ form a bipartite graph by Lemma \ref{PartitionEdge}~(iii), but indeed it is a clique with at least three vertices, a contradiction. 
\end{proof}
\begin{theorem}%
If $G$ is a $3$-$\gamma_{t}$-edge critical graph on $n$ vertices and with connectivity one, then $|E(G^{c})| < \left \lfloor \frac{n^{2}}{4} \right\rfloor$.
\end{theorem}
\begin{proof}%
If $\diam(G) = 3$, then we are done in \cite{Wang2011}. So we may assume that $\diam(G) = 2$. If $v$ is a cut vertex of $G$, then $v$ dominates $G$, and hence $v$ and one of its neighbor totally dominates $G$, this contradicts the fact that $\gamma_{t}(G) = 3$. 
\end{proof}

\begin{remark}
For the connectivity $\ell = 2$, Haynes \etal give a proof in \cite{MR2852589}, indeed, their proof covers the graphs of odd order, but they used the result about claw-free case in the proof, so we give a direct proof here.
\end{remark}
\begin{theorem}%
If $G$ is a $3$-$\gamma_{t}$-edge critical graph on $n$ vertices and with connectivity two, then $|E(G^{c})| < \left \lfloor \frac{n^{2}}{4} \right\rfloor$.
\end{theorem}
\begin{proof}%
If $\diam(G) =3$, then we are done in \cite{Wang2011}. So we may assume that $\diam(G) = 2$. Let $\{x, y\}$ be a minimum vertex cut. A vertex in $G - \{x, y\}$ is called {\em strong} if it joins to both $x$ and $y$, and {\em weak} if it joins to one of $x$ and $y$.

We state the following properties, they are very simple, so we omit their proofs, the readers can also find the proofs in \cite{MR2852589}.
\begin{enumerate}[(i)]%
\item $\{x, y\}$ dominates $G$ and every vertex in $G - \{x, y\}$ is either strong or weak;
\item $x$ and $y$ are nonadjacent;
\item the strong vertices in the same component of $G - \{x, y\}$ forms a clique;
\item there is at most one component of $G - \{x, y\}$ containing weak vertices;
\item the set of weak vertices is a clique;
\item there are precisely two components of $G - \{x, y\}$.
\end{enumerate}

Let $G_{1}$ and $G_{2}$ be the two components of $G - \{x, y\}$. Without loss of generality, assume that all the vertices in $G_{1}$ are strong. Let $A = V(G_{1}) \cup \{x, y\}$ and $B = V(G_{2})$.

The set $\{x, y\}$ is a minimum vertex cut, there are at least two edges in $[A, B]$. If $G_{2}$ is complete, then there are only one missing edge (say $xy$) in $A$ and $B$, and thus $|E(G^{c})| \leq |A| \times |B| - 1 < \left\lfloor \frac{n^{2}}{4} \right \rfloor$, we are done. So we may assume that $G_{2}$ is not complete. Let $uv$ be a missing edge in $G_{2}$. By the previous asserations, assume that $u$ is a strong vertex and $v$ is a weak vertex. Since $\{u, v\}$ does not dominate $G_{1}$, there exists a vertex $w$ such that $uw \mapsto v$ or $wv \mapsto u$. In both cases, the vertex $w$ has to dominate $G_{1}$, it follows that $w \in \{x, y\}$. If $wv \mapsto u$, then $wu \notin E(G)$, a contradiction. Then $uw \mapsto v$ and $uw$ is the quasi-edge of $uv$. Therefore, the quasi-edges of missing edges in $G_{2}$ are between $\{x, y\}$ and the strong vertices of $G_{2}$. If there are at least two weak vertices in $G_{2}$, then $|E(G^{c})| \leq |A| \times |B| -1 < \left \lfloor \frac{n^{2}}{4} \right \rfloor$, we are done. So there exists only one weak vertex, say $v$, in $G_{2}$. Assume that $yv \in E(G)$. Therefore, for any missing edge $uv$ in $B$, $xu$ is the quasi-edge of $uv$. There are two edges $yu, yv$ are not the quasi-edge of any missing edge in $B$, but there exist only one missing edge in $A$, so $|E(G^{c})| \leq |A| \times |B| - 1 < \left \lfloor \frac{n^{2}}{4} \right \rfloor$.
\end{proof}

\begin{theorem}%
If $G$ is a $3$-$\gamma_{t}$-edge critical graph on $n$ vertices and with connectivity 3, then $|E(G^{c})| < \left \lfloor \frac{n^{2}}{4} \right\rfloor$.
\end{theorem}
\begin{proof}%
\setcounter{claim}{0}
If $\diam(G) = 3$, then the theorem was proved in \cite{Wang2011}. So we may assume that $\diam(G) = 2$. Let $\{x, y, z\}$ be a minimum vertex cut of $G$.
\begin{claim}%
There are precisely two components of $G-\{x, y, z\}$.
\end{claim}
\begin{proof}
See \cite{MR2852589}.
\end{proof}
Let $G_{1}$ and $G_{2}$ be the two components of $G - \{x, y, z\}$, and let $V_{1}$ and $V_{2}$ be the vertex set of $G_{1}$ and $G_{2}$, respectively.
\begin{claim}\label{VClique}
For any two vertices $v, v'$ in the same component of $G - \{x, y, z\}$, if both $v$ and $v'$ dominate $\{x, y, z\}$, then $vv' \in E(G)$. 
\end{claim}
\begin{proof}%
Suppose that $vv' \notin E(G)$. Since $\{v, v'\}$ does not dominate the other component, there exists a vertex $w$ such that $vw \mapsto v'$ or $wv' \mapsto v$. In order to dominate the other component, $w \in \{x, y, z\}$, but both $v$ and $v'$ dominates $\{x, y, z\}$, a contradiction.
\end{proof}

It is easy to check the following asseration:
\begin{claim}\label{N}
For any vertices $v_{1} \in V_{1}$ and $v_{2} \in V_{2}$, we have either $N_{G}(v_{1}) \cap \{x, y, z\} \neq N_{G}(v_{2}) \cap \{x, y, z\}$ or both $v_{1}$ and $v_{2}$ dominates $\{x, y, z\}$.
\end{claim}

For $i = 1, 2$, let $S_{i}^{*}$ be the set of vertices in $\{x, y, z\}$ which dominates $V_{i}$.

\begin{claim}\label{S*}%
We may assume that $|S_{1}^{*} \cup S_{2}^{*}| = 3$.
\end{claim}
\begin{proof}%
We may assume, on the contrary, that there exists a vertex, say $z$, such that $\{zv_{1}, zv_{2}\} \subseteq E(G^{c})$, where $v_{1} \in V_{1}$ and $v_{2} \in V_{2}$. Since $\deg_{G}(v_{i}) \geq 3$ and $zv_{i} \notin E(G)$, we have $|V_{i}| \geq 2$ for $i=1, 2$. Since $d_{G}(v_{1}, v_{2}) = 2$, $v_{1}$ and $v_{2}$ have a common neighbor, say $x$, in $\{x, y, z\}$. Since $\{v_{1}, v_{2}\}$ does not dominate $z$, there exists a vertex $w$ such that $v_{1}w \mapsto v_{2}$ or $wv_{2} \mapsto v_{1}$. Without loss of generality, assume that $v_{1}w \mapsto v_{2}$.  Since $v_{1}w \in E(G)$ and $wv_{2} \notin E(G)$, it follows that $w = y$ and $yz \in E(G)$ and $y$ dominates $V_{2}$ except $v_{2}$. Since $\diam(G) = 2$ and $\{yv_{2}, zv_{2}\} \subseteq E(G^{c})$, it yields that $x$ dominates $V_{1}$. Hence $xy \notin E(G)$, for otherwise $\{x, y\}$ totally dominates $G$, which is a contradiction. If $uv$ is a missing edge in $V_{1}$, then there exists a vertex $w'$ such that $uw' \mapsto v$ or $w'v \mapsto u$, in both cases, $w'$ dominates $V_{2}$, so $w' = x$, but $\{xu, xv\} \subseteq E(G)$, a contradiction. Therefore, $V_{1}$ is a clique. The vertex $v_{2}$ has only one neighbor $x$ in $\{x, y, z\}$, by Claim~\ref{N}, for any vertex in $V_{1}$, it has one neighbor in $\{y, z\}$, and thus $\{y, z\}$ dominates $V_{1}$. 

Suppose that $v_{2}v'$ is a missing edge in $V_{2}$. Consider $G + v_{1}v'$. Since $\{v_{1}, v'\}$ does not dominate $v_{2}$, there exists a vertex $w^{*}$ such that $v_{1}w^{*} \mapsto v'$ or $w^{*}v' \mapsto v_{1}$. If $w^{*}v' \mapsto v_{1}$, then $w^{*}v_{1} \notin E(G)$ and $w^{*}v' \in E(G)$, it follows that $w^{*} =z$, but $\{z, v'\}$ does not dominate $v_{2}$, a contradiction. So we may assume that $v_{1}w^{*} \mapsto v'$, then $v_{1}w^{*} \in E(G)$ and $w^{*}v_{2} \in E(G)$, so $w^{*} = x$ and $x$ dominates $V_{2}$ except $v'$. Consider $G + v_{2}v'$, there exists a vertex $w$ in $\{x, y, z\}$ such that $v_{2}w \mapsto v'$ or $wv' \mapsto v_{2}$. If $v_{2}w \mapsto v'$, then $wv_{2} \in E(G)$ and $w = x$, but $\{x, v_{2}\}$ does not dominate $y$, a contradiction. If $wv' \mapsto v_{2}$, then $wv' \in E(G)$ and $wv_{2} \notin E(G)$, so $w = y$, but $\{y, v'\}$ does not dominate $x$, a contradiction. Hence we may assume that $v_{2}$ dominates $V_{2}$.

For any missing edge $u'v'$ in $V_{2}$, indeed, it is a missing edge in $N_{G_{2}}(y)$. Consider $G + u'v'$, quasi-edges of $u'v'$ lies in $[\{x\}, V_{2}\setminus \{v_{2}\}]$. 

Let $A = V_{1} \cup \{x, y\}$ and $B = V_{2} \cup \{z\}$. We associate every missing edge $u'v'$ in $V_{2}$ with one of its quasi-edge in $[\{x\}, V_{2}\setminus \{v_{2}\}]$; associate the missing edge in $[\{z\}, V_{2}]$ with edges in $[\{y\}, V_{2}]$; associate the missing edge $xy$ with $yz$. In addition, there is an additional edge $xv_{2}$, therefore, $|E(G^{c})| < \left\lfloor \frac{n^{2}}{4} \right\rfloor$ by Lemma \ref{QuasiClique}.
\end{proof}



\begin{claim}%
We may assume that $|V_{1}| = 1$ and $V_{2}$ is not a clique.
\end{claim}
\begin{proof}
By Claim~\ref{S*}, we have $S_{1}^{*} \cup S_{2}^{*} = \{x, y, z\}$, without loss of generality, assume that $\{x, y\} \subseteq S_{1}^{*}$. We may assume, on the contrary, that $|V_{1}| \geq 2$. Let $v_{1}$ be a neighbor of $z$ in $V_{1}$, hence $v_{1}$ dominates $\{x, y, z\}$. Let $Q$ be the set of vertices in $V_{2}$ which does not dominate $V_{2}$. For any vertex $v \in Q$, since $\{v_{1}, v\}$ does not dominate $V_{2}$, there exists a vertex $w_{v}$ such that $v_{1}w_{v} \mapsto v$ or $w_{v}v \mapsto v_{1}$. If $w_{v}v \mapsto v_{1}$, then $w_{v} \notin \{x, y, z\}$ and $w_{v} \in V_{2}$, but $\{w_{v}, v\}$ does not dominate $V_{1} \setminus \{v_{1}\}$, which is a contradiction. Therefore, for any vertex $v \in Q$, there exists a vertex $w_{v}$ in $\{x, y, z\}$ such that $v_{1}w_{v} \mapsto v$, and thus $w_{v}$ dominates $V_{2}$ except $v$. Note that for distinct vertices $v$ and $v'$ in $Q$, $w_{v}\neq w_{v'}$, therefore, $|Q| \leq 3$. 

If $|Q| = 3$, then $|V_{2}| \geq 4$, for otherwise $V_{2}$ is disconnected, a contradiction. Hence, for every vertex $x' \in V_{2} \setminus Q$, it dominates $V_{2} \cup \{x, y, z\}$, consequently, $\{x, x'\}$ totally dominates $G$, a contradiction.

If $|Q| = 2$, then $|V_{2}| \geq 3$, for otherwise $V_{2}$ is disconnected, a contradiction. Suppose that $z$ dominates $V_{1}$. By Claim~\ref{VClique}, $V_{1}$ is a clique. Without loss of generality, assume that for vertices $v, v' \in Q$, we have $w_{v} = x$ and $w_{v'} = y$. Let $A = V_{1} \cup \{x, y\}$ and $B = V_{2} \cup \{z\}$. We associate the missing edge $vv'$ in $V_{2}$ with $zv_{1}$, associate missing edges between $z$ and $V_{2}$ with edges in $[\{y\}, V_{2}]$. Now, there are at least two edges in $[\{x\}, V_{2}]$ but there are at most one missing edge (say, $xy$) in $V_{1} \cup \{x, y\}$, hence $|E(G^{c})| < \left\lfloor \frac{n^{2}}{4} \right\rfloor$ by Lemma~\ref{QuasiClique}, we are done. So we may assume that $z$ dominates $V_{2}$, then we have $\{w_{v}, w_{v'}\} = \{x, y\}$. Also let $A = V_{1} \cup \{x, y\}$ and $B = V_{2} \cup \{z\}$. For any missing edge $e$ in $V_{1}$, quasi-edges of $e$ lies in $[\{z\}, V_{1}]$, we associate the missing edge $e$ with one of its quasi-edges, associate the missing edge in $V_{2}$ with one edge in $[\{x, y\}, V_{2}]$, associate one edge in $[\{x, y\}, V_{2}]$ with the possible missing edge $xy$. In the final, there are at least $2|V_{2}| - 2 - 2 \geq 2$ additional edges, hence $|E(G^{c})| < \left\lfloor \frac{n^{2}}{4} \right\rfloor$ by Lemma~\ref{QuasiClique}.

Next, we consider the case $|Q| = 0$, \ie $V_{2}$ is a clique. If $z$ dominates $V_{2}$, then $V_{2} \cup \{z\}$ is a clique. For any missing edge in $V_{1}$, its quasi-edges lies in $[\{z\}, V_{1}]$. There are at least two edges in $[\{x, y\}, V_{2}]$ and at most one other missing edge (say $xy$) in $V_{1} \cup \{x, y\}$, then there are at least one additional edge, hence $|E(G^{c})| \leq |V_{1} \cup \{x, y\}| \times |V_{2} \cup \{z\}| -1 < \left\lfloor \frac{n^{2}}{4} \right\rfloor$. If $z$ does not dominate $V_{2}$, then it dominates $V_{1}$ and $V_{1}$ is a clique by Claim~\ref{VClique}. 

Now, if $|V_{1}| \geq 2$, then we may assume that $V_{1}$ is dominated by every vertex in $\{x, y, z\}$, and both $V_{1}$ and $V_{2}$ are clique. Clearly, if $|V_{1}| = 1$, then $V_{1}$ is also dominated by every vertex in $\{x, y, z\}$. In order to prove the claim by contradiction, we may assume that both $V_{1}$ and $V_{2}$ are cliques regardless the size of $V_{1}$. Let $A' = V_{1} \cup \{x, y, z\}$ and $B' = V_{2}$. There are at least three edges between $A'$ and $B'$ since $\{x, y, z\}$ is a minimum vertex cut, and there are at most three missing edges in $A'$. If $|E(G^{c})| = \left\lfloor \frac{n^{2}}{4} \right\rfloor$, then $\{x, y, z\}$ is independent and $|[\{x, y, z\}, V_{2}]| = 3$, the subgraph formed by the missing edges in $A'$ contains a triangle, which contradicts with Lemma \ref{PartitionEdge}.
\end{proof}

For every missing edge in $V_{2}$, we associate it with an unique quasi-edge of it, and denote this set by $Q_{e}$.

\begin{claim}%
We may assume that there are at least two edges in $\{x, y, z\}$.
\end{claim}
\begin{proof}
Suppose that the subgraph induced by $\{x, y, z\}$ has at most one edge, without loss of generality, let $z$ be an isolated vertex in this subgraph. Let $Q$ be the set of vertices in $V_{2}$ which dominates $\{x, y, z\}$. Then $Q$ is a clique by Claim \ref{VClique}. Let $R = V_{2} \setminus Q$, $B = \{z\} \cup R$ and $A = V\setminus B$. 

Let $v$ be an arbitrary vertex in $R$. If $zv \in E(G)$, then $zv \notin Q_{e}$, for otherwise, $v$ has to dominate $\{x, y\}$ and $v \in Q$, a contradiction. If $zv \notin E(G)$, then for every edge in $[\{x, y\}, \{v\}]$ (note that $[\{x, y\}, \{v\}] \neq \emptyset$), it can not belong to $Q_{e}$, for otherwise, $v$ has to dominate $z$ and $zv \in E(G)$, a contradiction. Hence, for any vertex $v$ in $R$, there is at least one edge in $[\{v\}, \{x, y, z\}]$ such that it is not in $Q_{e}$.

For any missing edge $e$ in $R$, quasi-edges of $e$ lies in $[\{x, y\}, R]$. If $xy \notin E(G)$, then $R$ is a clique, moreover, $[Q, R] \neq \emptyset$ since $V_{2}$ is not a clique and $V_{2}$ is connected. So, if $xy \notin E(G)$, then we associate $xy$ with one edge in $[Q, R]$. If $zv$ is a missing edge in $B$, we associate an edge in $[\{x, y\}, \{v\}]\setminus Q_{e}$ with $zv$. We associate the missing edges in $[\{v_{1}\}, Q]$ with edges in $[\{z\}, Q]$. There is an additional edge $zv_{1}$, and hence $|E(G^{c})| < \left\lfloor \frac{n^{2}}{4} \right\rfloor$ by Lemma~\ref{QuasiClique}.
\end{proof}
Let $B = V_{2}$ and $A = V \setminus B$. Without loss of generality, let $xyz$ be a path, $xz$ may be a missing edge.

Since $V_{2}$ is not a clique, $|V_{2}| \geq 2$; moreover, $|V_{2}| \geq 3$ since $V_{2}$ is connected. If $|V_{2}| = 3$, then $V_{2}$ has only one missing edge. But $|[A, B]| \geq 3$, and there are at most two missing edges in $A$ and $B$, so $|E(G^{c})| < \left\lfloor \frac{n^{2}}{4} \right\rfloor$, a contradiction. Obviously, the degree of $y$ is at least four. If $|V_{2}| = 4$, then $|E(G^{c})| = \binom{8}{2} - |E(G)| \leq 28- 13 < \left\lfloor \frac{n^{2}}{4} \right\rfloor$. So we may assume that $|V_{2}| \geq 5$.

{\bf Case 1}. $xz \notin E(G)$.

We may assume that there is at most one edge in $[\{x, y, z\}, V_{2}]$ which is not in $Q_{e}$, for otherwise we associate one of them with the missing edge $xz$, there is at least one additional edge, hence $|E(G^{c})| < \left\lfloor \frac{n^{2}}{4} \right\rfloor$ by Lemma~\ref{QuasiClique}.

The set $\{x, y, z\}$ is a minimum vertex cut, so each of $x$ and $z$ has at least one neighbor in $V_{2}$. Without loss of generality, assume that $xw \in [\{x\}, V_{2}]$ such that it is in $Q_{e}$ and $xw \mapsto w_{1}$. Hence $wz \in E(G)$ since $w$ has to dominate $z$. Suppose further that $wz$ is also in $Q_{e}$ and $wz \mapsto w_{2}$. Note that $w_{1} \neq w_{2}$ and $\{zw_{1}, xw_{2}\} \subseteq E(G)$. But $xw_{2}$ is not in $Q_{e}$ since $\{x, w_{2}\}$ does not dominate $z$, and $zw_{1}$ is also not in $Q_{e}$ since $\{z, w_{1}\}$ does not dominate $x$. Now, there are at least two edges in $[\{x, y, z\}, V_{2}]$ which are not in $Q_{e}$, a contradiction. Hence, $zw \in E(G)$ but $zw \notin Q_{e}$ and every edge in $[\{x, y, z\}, V_{2}] \setminus \{zw\}$ is in $Q_{e}$. By the previous argument, we have $[\{x, z\}, V_{2}] = \{xw, zw\}$. As $\diam(G) = 2$, every vertex in $V_{2}$ has at least one neighbor in $\{x, y, z\}$, so $[\{y\}, V_{2} \setminus \{w\}]$ is full and $[\{y\}, V_{2} \setminus \{w\}] \subseteq Q_{e}$. For every edge $yw'$ in $[\{y\}, V_{2} \setminus \{w\}]$, it is a quasi-edge of missing edge in $V_{2}$, hence $yw'\mapsto w$ and $ww' \notin E(G)$, moreover, $w$ is isolated in $V_{2}$, which is a contradiction.

{\bf Case 2}. $xz \in E(G)$ and hence $A$ is a clique.

We associate every missing edge in $B$ with a unique quasi-edge in $Q_{e}$, hence, $4(n-4) \geq |E(G^{c})|$. If $|E(G^{c})| \geq \left\lfloor \frac{n^{2}}{4} \right\rfloor$, then $n = 9$ and $Q_{e} = [A, B]$ and Lemma \ref{PartitionEdge} (i)--(iii) holds.

For any vertex in $V_{2}$, it does not dominate $V_{2}$, otherwise, chooose a neighbor of it in $\{x, y, z\}$, we obtain a two vertex set totally dominates $G$, a contradiciton. 

Let \[X = \{v \in V_{2} \mid \text{$v$ has an odd number of neighbors in $\{x, y, z\}$}\}\]
and \[Y = \{v \in V_{2} \mid \text{$v$ has an even number of neighbors in $\{x, y, z\}$}\}.\]
Then $X \cup Y = V_{2}$, and the missing edges in $V_{2}$ form a bipartite graph $H$ with bipartition $(X, Y)$ by Lemma \ref{PartitionEdge}. Let $m^{*}=\min\{|X|, |Y|\}$. Since $|Q_{e}| = |[A, B]| \geq |V_{2}| = 5$, there are at least five missing edges in $V_{2}$, \ie $H$ has at least five edges, so $m^{*} = 2$. Hence, there are at least $|X| + 2|Y| \geq |V_{2}| + m^{*} = 7$ edges in $[A, B]$, but there are at most $|X| \times |Y| = 6$ edges in $H$, \ie there are at most 6 missing edges in $V_{2}$, a contradiction.
\end{proof}
\vskip 3mm \vspace{0.3cm} \noindent{\bf Acknowledgments.} The first author was supported by NSFC (11101125), the third
author is supported by the Discovery Grant (144073) of Natural
Sciences and Engineering Research Council of Canada.

\end{document}